\documentclass[12pt,reqno]{amsart}
\usepackage{amsaddr}

\usepackage{a4wide}
\usepackage{amsthm}
\usepackage{amsmath}
\usepackage{amssymb}
\usepackage[all,cmtip]{xy}
\usepackage{color}
\usepackage[T1]{fontenc}
\usepackage{enumerate}

\theoremstyle{plain}
\newtheorem{thm}{Theorem}
\newtheorem{lem}[thm]{Lemma}
\newtheorem{prop}[thm]{Proposition}

\theoremstyle{definition}
\newtheorem{defn}[thm]{Definition}
\newtheorem{rem}[thm]{Remark}

\DeclareMathOperator{\id}{id}

\DeclareMathOperator{\supp}{Supp}

\newcommand{\SkewRing}{A \star_\alpha G}

\begin{document}

\title[Simple Semigroup Graded Rings]{Simple Semigroup Graded Rings}


\author{Patrik Nystedt}
\address{University West, Department of Engineering Science, SE-46186
  Trollh\"{a}ttan, Sweden}
\email{Patrik.Nystedt@hv.se}

\author{Johan \"{O}inert}
\address{Centre for Mathematical Sciences, P.O. Box 118, Lund University, SE-22100 Lund, Sweden}
\email{Johan.Oinert@math.lth.se}

\subjclass[2010]{16W50, 16D25, 16U70, 16S35}
\keywords{semigroup graded ring, partial skew group ring, simplicity.}

\begin{abstract}
We show that if $R$ is a, not necessarily unital, 
ring graded by a semigroup $G$
equipped with an idempotent $e$ such that
$G$ is cancellative at $e$,
the non-zero elements of
$eGe$ form a hypercentral group and
$R_e$ has a non-zero idempotent $f$,
then $R$ is simple if and only if 
it is graded simple and 
the center of the corner subring $f R_{eGe} f$ is a field.
This is a generalization of a result of E. Jespers' on the 
simplicity of a unital ring graded by a hypercentral group.
We apply our result to partial skew group rings
and obtain necessary and sufficient conditions for the
simplicity of a, not necessarily unital, partial skew group ring 
by a hypercentral group. 
Thereby, we generalize a very recent result of D. Gon\c{c}alves'.
We also point out how E. Jespers' result immediately implies
a generalization of a simplicity result,
recently obtained by A. Baraviera, W. Cortes and M. Soares, 
for crossed products by twisted partial actions.
\end{abstract}

\maketitle

\pagestyle{headings}

\section{Introduction}%

Suppose that $R$ is an associative,
not necessarily unital, ring
and $G$ is a semigroup, i.e. a non-empty
set equipped with an associative binary operation 
$G \times G \ni (g,h) \mapsto gh \in G$.
Recall that $R$ is called 
\emph{$G$-graded} if there for each $g \in G$
is an additive subgroup $R_g$ of $R$
such that $R = \oplus_{g \in G} R_g$
and the inclusion $R_g R_h \subseteq R_{gh}$ holds
for all $g,h \in G$.

The investigation of semigroup graded rings has been carried out by
many authors, see e.g.
\cite{Abrams96,Bell96,Chanyshev90,Chick87,Clase96,Dascalescu01,Gardner75,Ignatov82,
Karpilovsky92,Krempa85,Munn88,Ponizovskii87,Puczylowski91,Wauters87}.
For an excellent overview of the theory of semigroup graded rings,
we refer the reader to A. V. Kelarev's extensive book \cite{Kelarev02}, and the references therein.

Since many ring constructions are special cases of semigroup graded rings,
e.g. monomial rings, crossed products, skew polynomial rings,
twisted semigroup rings, skew power series rings,
edge and path algebras, generalized matrix rings,
incidence algebras and category graded rings,
the theory of semigroup graded rings can be
applied to the study of other less general constructions, giving new results
for several constructions simultaneously, and unifying theorems obtained
earlier. 

An important problem in the investigation of semigroup graded rings is
to explore how properties of the whole ring $R$ are connected
to properties of subrings $R_H = \oplus_{g \in H} R_g$ where $H$ runs over subsemigroups
of $G$. Many results of this sort are known for finiteness conditions,
nil and radical properties, semisimplicity, semiprimeness
and semiprimitivity
(see the references in \cite{Kelarev95} and \cite{Kelarev98}). 

The aim of this article is to establish
a similar result (see Theorem \ref{maintheorem}) for simple semigroup graded rings.
Recall that an (two-sided) ideal $I$ of a $G$-graded ring $R$
is said to be a \emph{graded ideal} (or \emph{$G$-graded ideal})
if $I = \oplus_{g \in G} (I \cap R_g)$ holds.
A $G$-graded ring $R$ is called \emph{graded simple} (or \emph{$G$-graded simple})
if $R$ and $\{0\}$ are its only graded ideals.
Clearly, graded simplicity is
a necessary condition for simplicity.
In the case when $G$ is a \emph{hypercentral} group, 
that is when every non-trivial factor group of $G$
has non-trivial center, E. Jespers \cite{jespers1993} 
has determined precisely when unital $G$-graded rings are simple.

\begin{thm}[E. Jespers \cite{jespers1993}]\label{jespers}
If $G$ is a hypercentral group and
$R$ is a $G$-graded unital ring,
then $R$ is simple if and only if
$R$ is graded simple and the center
of $R$ is a field.
\end{thm}

Suppose that $R$ is a ring graded by a
semigroup equipped with a \emph{zero element},
i.e. an element $\theta \in G$ satisfying
$\theta g = g \theta = \theta$, for $g \in G$.
This implies that whenever $R_{\theta}$ is non-zero
and there is a non-zero $g \in G$ such that $R_g$
is non-zero, $R_{\theta}$ is a non-trivial
ideal of $R$ and, hence, $R$ is not simple.
We therefore, throughout this article, make the assumption
that {\it if $G$
has a zero element $\theta$, then $R_{\theta} = \{ 0 \}$.}
To state our generalization of Theorem \ref{jespers},
we need to introduce the following semigroup notion.
If $G$ is a semigroup and $e$ is a non-zero idempotent in $G$,
then we say that $G$ is {\it cancellative at} $e$
if for any $a,b,x,y \in G$, a relation of the form
$ea x be = ea y be \neq \theta$, implies that $x = y$.
Clearly, all groups are cancellative at identity elements.
Also, all semigroups induced by groupoids are
cancellative at identity elements (see Remark \ref{remarkcategory}).

\begin{thm}\label{maintheorem}
If $R$ is a ring graded by a semigroup $G$
equipped with a non-zero idempotent $e$ such that
$G$ is cancellative at $e$,
the non-zero elements of $eGe$ form a hypercentral group and
$R_e$ has a non-zero idempotent $f$,
then $R$ is simple if and only if 
it is graded simple and 
the center of the corner subring $f R_{eGe} f$ is a field.
\end{thm}

Here we would like to make two remarks.
First of all, note that the above result might at first glance seem more general than it 
actually is. In fact, graded simplicity of $R$
often forces the semigroup $G$ to be 
an inverse semigroup (see Remark \ref{inversesemigroup}).
Secondly, up until recently, the authors of the present article (and presumably also 
the authors of e.g. \cite{BCS}, \cite{crow2005}) were unaware 
of the existence of E. Jespers' simplicity results in \cite{jespers1989} 
and \cite{jespers1993}, as well as A. D. Bell's in \cite{Bell}.
The technique used to prove the simplicity criterion 
for skew group rings in \cite{oinertarxiv11} is slightly 
different from the one used in \cite{jespers1989} and \cite{jespers1993},
but can in fact, after minor adjustments, be used to give a more 
direct proof of the main result of \cite{jespers1989}.

This article is organized as follows.

In Section \ref{mainsection}, we
recall the relevant definitions
concerning semigroup graded rings and prove Theorem \ref{maintheorem}.

In Section \ref{applicationpartialskewgroupring},
we apply Theorem \ref{maintheorem} to partial skew
group rings. 
Partial group actions were introduced by 
R. Exel in the context of crossed product C*-algebras \cite{Exel94}.
A decade later, the investigation of its algebraic counterpart (the partial skew group rings) began by \cite{DokuchaevExel05}.
We obtain necessary and sufficient conditions for the
simplicity of, not necessarily unital, partial skew group rings
by hypercentral groups. 
Thereby, we generalize a very recent 
result of D. Gon\c{c}alves' \cite{Goncalves13}.
At the end of the section,
we also point out how E. Jespers' result immediately implies
a generalization of a simplicity result,
recently obtained by A. Baraviera, W. Cortes and M. Soares, 
for crossed products by twisted partial actions.

\section{Semigroup Graded Rings}\label{mainsection}%

At the end of this section, we show Theorem \ref{maintheorem}.
To this end, we show a series of results concerning 
ideals in semigroup graded rings
(see Lemma \ref{unit}, Lemma \ref{foreachi}, 
Proposition \ref{firstprop} and Proposition \ref{secondprop}).

We begin by fixing the notation.
Throughout this section, $R$ denotes
a ring graded by a semigroup $G$.
Take $r \in R$.
There are unique
$r_g \in R_g$, for $g \in G$,
such that all but finitely many of them
are zero and $r = \sum_{g \in G} r_g$.
We let the \emph{support} of $r$,
denoted by $\supp(r)$, be the set of 
$g \in G$ such that $r_g \neq 0$.
The cardinality of $\supp(r)$ is denoted by $|\supp(r)|$.
The element $r$ is called
\emph{homogeneous} if $|\supp(r)| \leq 1$.
If $r \in R_g \setminus \{ 0 \}$, for some $g \in G$,
then we write $\deg(r)=g$.

\begin{lem}\label{unit}
Let $G$ be a group and 
$R$ a unital $G$-graded ring.
If $R$ is graded simple and $I$ is a non-zero
$G/Z(G)$-graded ideal of $R$, then for each 
non-zero $r \in I$, there is a non-zero $r' \in 
I \cap R_{Z(G)} \cap Z(R)$
with $| \supp(r') | \leq | \supp(r) |$.
If, in addition, the ring $Z(R)$ is a field, 
then $R$ is $G/Z(G)$-graded simple.
\end{lem}

\begin{proof}
This is Proposition 4 of \cite{jespers1993}.
We show this result by a different method.
Let $I$ be a non-zero $G/Z(G)$-graded ideal of $R$ and take a non-zero $r\in I$.
Choose $g\in G$ such that $r_g \neq 0$.
Since $R$ is graded simple, we get that
$R r_g R = R$. In particular, we get that
$1 = \sum_{i=1}^n s_i r_g t_i$ for some
homogeneous $s_i,t_i \in R$.
Therefore, there is $j \in \{ 1,\ldots,n \}$
such that $s_j r_g t_j \neq 0$
and $\deg(s_j r_g t_j) = e$. 
By replacing $r$ with $s_j r t_j$
we can assume that $r_e \neq 0$.
Since $I$ is $G/Z(G)$-graded, 
the $Z(G)$-degree part of $r$ belongs to $I$
and is non-zero (since $r_e \neq 0$).
Therefore, we can assume that $r$
is a non-zero element belonging to $I \cap R_{Z(G)}$.

Now put $J = \{ s_e \mid s \in R r R, \ \supp(s) \subseteq \supp(r) \}$.
We want to show that $1\in J$.
Note that $J$ is a non-zero ideal of $R_e$ and hence
that $R J R$ is a non-zero graded
ideal of $R$.
By graded simplicity of $R$
we get that $R J R = R$.
Thus, there are $s^{(1)},\ldots,s^{(n)} \in R r R$
and $v_i,w_i \in R$, for $i \in \{1,\ldots,n\}$,
such that $1 = \sum_{i=1}^n v_i s^{(i)}_e w_i$
and $\supp( s^{(i)} ) \subseteq \supp(r)$.
This implies that $\deg(v_i) \deg(w_i) = e$
for $i \in \{1,\ldots,n\}$.
Put $s =  \sum_{i=1}^n v_i s^{(i)} w_i$.
Then $s \in I$ and since 
$\supp(s^{(i)}) \subseteq \supp(r) \subseteq R_{Z(G)}$,
we get that $\supp(s) \subseteq 
\cup_{i=1}^n \deg(v_i) \supp( s^{(i)} ) \deg(w_i)
\subseteq \cup_{i=1}^n \deg(v_i) \deg(w_i) \supp( s^{(i)} )
\subseteq \cup_{i=1}^n e \supp(r) = \supp(r)$.
Therefore, $1 = \sum_{i=1}^n v_i s^{(i)}_e w_i = s_e \in J$.

Now pick a non-zero element $r' \in I$
with $|\supp(r')|$ minimal. By the above,
we can assume that $r'_e = 1$
and that $r' \in I \cap R_{Z(G)}$.
Take $g \in G$ and $t \in R_g$.
Since $r'_e=1$ and $\supp(r') \subseteq Z(G)$,
we get that $|\supp(r't - tr')| < |\supp(r')|$.
By the assumptions on $r'$ we get that
$\supp(r't - tr') = \emptyset$ and hence
that $r't - tr' = 0$.
Therefore $r' \in Z(R)$.
\end{proof}

\begin{rem}\label{ascending}
Recall that if $G$ is a group with identity
element $e$, then the ascending central series 
of $G$ is the sequence of subgroups $Z_i(G)$, for
non-negative integers $i$, defined recursively by
$Z_0(G) = \{ e \}$ and, given $Z_i(G)$, for
some non-negative integer $i$, 
$Z_{i+1}(G)$ is defined to be the set of 
$g \in G$ such that for every $h \in G$,
the commutator $[g,h]=g h g^{-1} h^{-1}$
belongs to $Z_i(G)$.
For infinite groups this process can be
continued to infinite ordinal numbers
by transfinite recursion. For a limit ordinal
$O$, we define $Z_O(G) = \cup_{i < O} Z_i(G)$.
If $G$ is hypercentral, then
$Z_O(G) = G$ for some limit ordinal $O$.
For the details concerning this 
construction, see \cite{mal49}.
\end{rem}

\begin{lem}\label{foreachi}
If $G$ is a hypercentral group and
$R$ is a $G$-graded ring with the 
property that for each $i < O$
the ring $R$ is $G/Z_i(G)$-graded simple,
then $R$ is simple.
\end{lem}

\begin{proof}
Take a non-zero ideal $J$ of $R$
and a non-zero $a \in J$.
We show that $\langle a \rangle = R$.
Since $\cup_i Z_i(G) = G$ and $\supp(a)$
is finite, we can conclude that
there is some $i$ such that $\supp(a) \subseteq Z_i(G)$.
Then $\langle a \rangle$ is a non-zero 
$G / Z_i(G)$-graded ideal of $R$.
Since $R$ is $G/Z_i(G)$-graded simple,
we get that $\langle a \rangle = R$, which shows that $J=R$.
\end{proof}

\begin{prop}\label{firstprop}
If $R$ is a simple ring graded by a semigroup $G$
equipped with a non-zero idempotent $e$ 
and $R_e$ 
contains a non-zero idempotent $f$,
then the center of the corner subring $f R_{eGe} f$ is a field.
\end{prop}

\begin{proof}
Take a non-zero $x$
in $Z(f R_{eGe} f)$. 
Since the ideal $R x R$ of $R$ is non-zero and $R$ is simple, 
we get that $R x R = R$.
In particular, $f$ equals a finite
sum of elements of the form $y_i x z_i$, for $i \in \{1,\ldots,n\}$,
where each $y_i$ and each $z_i$ is homogeneous.
Hence $f = f \cdot f \cdot f = 
\sum_{i=1}^n f y_i f x f z_i f$
so we may assume that $x_i,y_i \in f R_{eGe} f$ for all $i$.
But since $x$ belongs to $Z(f R_{eGe} f)$
we get that $f = w x = x w$ for some 
$w \in f R_{eGe} f$. 
All that is left to show now is that
$w \in Z(f R_{eGe} f)$.
Take $v \in f R_{eGe} f$.
Then, since $x$ commutes with $v$, we get that
$w v = w v f = w v x w = w x v w = f v w = v w$.
\end{proof}

\begin{prop}\label{secondprop}
Suppose that $G$ is a semigroup and $R$
is a ring graded by $G$. If $R$ is graded simple
and there is a non-zero idempotent $e \in G$
such that $G$ is cancellative at $e$,
the non-zero elements of $eGe$ form a hypercentral group, 
$R_e$ contains a non-zero idempotent $f$ 
and the center of the corner subring $f R_{eGe} f$ is a field,
then $R$ is simple. 
\end{prop}

\begin{proof}
Let $H$ denote the group of non-zero elements of $eGe$ and put $S = f R_H f$.
We claim that $S$ is simple.
Assume for a moment that the claim holds.
We show that $R$ is simple.
Take a non-zero ideal $I$ of $R$
and a non-zero $x \in I$. 
Take $g \in G$ such that $x_g \neq 0$.
Let $J$ denote the smallest two-sided ideal of $R$
containing $x_g$, i.e. 
$J = {\Bbb Z}x_g + R x_g + x_g R + R x_g R$.
Since $J$ is a graded ideal and $R$ is graded simple,
we get that $f \in J$.
Since $f = f^3$, this implies that 
$f \in fJf$ and thus we can write
$f = \sum_{i=1}^n f y_i x_g z_i f$ for some
homogeneous $y_i,z_i \in R$.
From the fact that $f$ is non-zero it now
follows that there is $j \in \{ 1,\ldots,n \}$
such that $f y_j x_g z_j f$ is non-zero.
Now put $x' = f y_j x z_j f$.
By the construction of $x'$ it follows that 
$x' \in I \cap S$. Since $G$ is cancellative at $e$ 
it also follows that $x'$ is non-zero.
We thus get that $I \cap S \neq \{ 0 \}$.
But since $S$ is simple, we get that
$I  \cap S = S$, or, equivalently,
that $S \subseteq I$.
This implies, in particular, that $f \in I$.
Since $f$ is homogeneous and $R$ is graded simple, 
we thus get that $I \supseteq RfR = R$. Hence $I=R$.

Now we show the claim in the beginning of the proof,
i.e. that $S$ is simple.
Let $Z_i(H)$, for $i \geq 0$,
be the ascending central series of $H$
(see Remark \ref{ascending}).
By induction over $i$ we now show that for each $i \in I$, the ring
$S$ is $H / Z_i(H)$-graded simple.

First we show the base case: $i=0$. 
Since $H/Z_0(H) = H / \{ e \} = H$,
we need to show that $S$ is $H$-graded simple.
Suppose that $J$ is a non-zero $H$-graded ideal of $S$.
Then $RJR$ is a non-zero $G$-graded ideal of $R$.
Since $R$ is graded simple, we get that $RJR = R$.
In particular, we get that $f \in RJR$
and hence $f= \sum_{i=1}^n y_i x_i z_i$ for some $x_i \in J$ and homogeneous $y_i, z_i \in R$.
From this we get that $f= f^3 = \sum_{i=1}^n f y_i f x_i f z_i f \in SJS \subseteq J$.
Since $f$ is a multiplicative identity of $S$,
we get that $J = S$.

Now we show the induction step.
Suppose that the statement is true for some $i$,
i.e. that $S$ is $H / Z_i(H)$-graded simple.
By Lemma \ref{unit}, we get that
$S$ is $\frac{ H/Z_i(H) }{ Z(H/Z_i(H)) }$-graded simple.
Since the center of $H / Z_i(H)$ equals 
$Z_{i+1}(H)/Z_i(H)$
we get that $S$ is $\frac{ H/Z_i(H) }{ Z_{i+1}(H)/Z_i(H) }$-graded simple,
i.e. that $S$ is $H/Z_{i+1}(H)$-graded simple 
and the induction step is complete.

By Lemma \ref{foreachi}, we get that $S$ is simple.
\end{proof}

\subsection*{Proof of Theorem \ref{maintheorem}}
This follows immediately from 
Propositions \ref{firstprop} and
\ref{secondprop}.
\hfill $\qed$ 

\begin{rem}
Theorem \ref{jespers} (and hence Theorem \ref{maintheorem})
can not be generalized to arbitrary groups (or semigroups).
Indeed, let $G$ denote the free group on two generators,
let $K$ be a field and suppose that $R=K[G]$ denotes
the group ring of $G$ over $K$.
Since $G$ is an ICC-group, i.e. each non-identity element has an infinite conjugacy class,
it follows from \cite[Proposition 5.3]{oinertarxiv11} that $Z(R)=K$ is a field.
It is not difficult to see that $R=K[G]$ is graded simple with respect to its natural $G$-gradation.
However, $R$ is not simple since, for any non-identity $g \in G$,
the ideal $R(1-g)R$ is non-trivial. 
\end{rem}

\begin{rem}\label{inversesemigroup}
By assuming that a $G$-graded ring $R$ is graded simple,
one often imposes restrictions on the semigroup $G$.
In fact, if we suppose that the following rather mild
conditions are satisfied
\begin{itemize}

\item[(i)] for each non-zero $g \in G$, there is $p,q \in G$
such that $R_p R_g R_q$ is non-zero, and

\item[(ii)] there is a non-zero idempotent $e$ in $G$,

\end{itemize}
then $G$ is an inverse semigroup.
In fact, take $g,h \in G$ to be non-zero. By (i), we get that
$R R_g R$ is a non-zero graded ideal of $R$.
Then, by graded simplicity of $R$, 
we get that $R_h \subseteq R R_g R$.
By (i) again, we get that $R_h$ is non-zero.
Hence, there are $p,q \in G$ such that
$R_p R_g R_q$ is non-zero and $pgq = h$.
In other words, we get that $h \in GgG$
for all non-zero $g,h \in G$.
Thus, $G$ is simple.
This, in combination with the existence of
the non-zero idempotent $e$, implies that 
$G$ is an inverse semigroup 
(see Theorem 3 in \cite{Mcfadden61}).
\end{rem}

\begin{rem}\label{remarkcategory}
Suppose that $G$ is a small category, i.e. with the property that its morphisms form a set.
Let the domain and codomain of a morphism $g$ in $G$ be denoted by
$d(g)$ and $c(g)$ respectively.
Note that every category $G$ can be viewed as a
semigroup if we adjoin a \emph{zero element} $\theta$
with the property that $g \theta = \theta g = \theta$,
for $g \in G$, and $gh = \theta$ whenever
$gh$ is undefined in the category, i.e. when $d(g) \neq c(h)$.
As a consequence, category graded rings (in the sense of \cite{lu04})
can be viewed as semigroup graded rings. 
Hence, Theorem \ref{maintheorem} is applicable in this situation as well.
In particular, note that if $G$ is a groupoid, i.e.
a category in which all morphisms are invertible,
then the semigroup defined by $G$ (as above) 
is cancellative at all identity morphisms of $G$.
\end{rem}

\section{Applications to Partial Skew Group Rings}\label{applicationpartialskewgroupring}

In this section, we apply Theorem \ref{maintheorem}
to partial skew group rings.
We generalize a recent result
by D. Gon\c{c}alves \cite{Goncalves13} to partial skew group rings by
hypercentral groups over rings with local units  
(see Theorem \ref{Simplicity}). 
At the end of this section,
we also point out how E. Jespers' result immediately implies
a generalization of a simplicity result,
recently obtained by A. Baraviera, W. Cortes and M. Soares, 
for crossed products by twisted partial actions
(see Remark \ref{partialcrossedproducts}).
First we recall the definition of a partial 
skew group ring.

\begin{defn}
Let $G$ be a group with neutral element $e$ and let $A$ be a ring.
A \emph{partial action} $\alpha$
of $G$ on $A$ is a collection of ideals $\{D_g\}_{g\in G}$
of $A$ and a collection of ring isomorphisms $\alpha_g : D_{g^{-1}} \to D_g$ such that
for all $g,h \in G$ and every $x\in D_{h^{-1}} \cap D_{(gh)^{-1}}$,
the following three relations hold:
\begin{center}
(i) $\alpha_e = \id_{A}$; \quad
(ii) $\alpha_g(D_{g^{-1}} \cap D_h) = D_g \cap D_{gh}$; \quad
(iii) $\alpha_g(\alpha_h(x)) = \alpha_{gh}(x)$.
\end{center}
The \emph{partial skew group ring} 
$A \star_\alpha G$, associated with the partial action above,
is defined as the set of all finite formal sums 
$\sum_{g\in G} a_g \delta_g$, where for each $g\in G$, $a_g \in D_g$ and $\delta_g$ is a symbol.
Addition is defined in the obvious way and multiplication is defined as the linear extension of the rule
$(a_g \delta_g)(b_h \delta_h)=\alpha_g(\alpha_{g^{-1}}(a_g)b_h) \delta_{gh}$
for $g,h\in G$, $a_g \in D_g$ and $b_h \in D_h$.
It is easy to check that if we put
$(\SkewRing)_g = D_g \delta_g$, for $g \in G$,
then this defines a gradation on the ring $\SkewRing$.
Clearly, each classical skew group ring
(see e.g. \cite{FisherMontgomery78}) is 
a partial skew group ring where $D_g = A$ for all $g \in G$.
\end{defn}

\begin{defn}
Recall from \cite{anh87} that a ring $R$ has
{\it local units} if  
there exists a set $E$ of idempotents
in $R$ such that, for every finite subset $X$ of $R$,
there exists an $f \in E$
such that $X \subseteq fRf$. From this it follows that $x=fx=xf$ holds for each $x\in X$.
In that case, we will refer to $E$ as a 
{\it set of local units} for $R$
and to $f$ as a \emph{local unit} for the subset $X$.
\end{defn}

\begin{rem}
Rings with local units occur widely in mathematics, often in algebra
(e.g. von Neumann regular rings \cite[Example 1]{anh87} and
Leavitt path algebras \cite[Lemma 1.6]{Abrams05}),
functional analysis (e.g. algebras of complex-valued
functions with compact support) 
and category theory (e.g. \cite[Remark 1]{harada1973} or \cite{fuller1978}).
\end{rem}

\begin{rem}
A partial skew group ring $A \star_\alpha G$ need not 
in general be associative (see \cite[Example 3.5]{DokuchaevExel05}).
However, if each $D_g$, for $g \in G$, has  
local units, then, in particular, each $D_g$, for $g \in G$,
is an idempotent ring, i.e. $D_g^2 = D_g$, which
by \cite[Corollary 3.2]{DokuchaevExel05},
ensures that $\SkewRing$ is associative.
In that case, the set $E \delta_e = \{f\delta_e \mid f\in E\}$ 
is a set of local units for $\SkewRing$,
if $E$ is a set of local units for $A$.
\end{rem}

\begin{defn}
If $\SkewRing$ is a partial skew group ring, then
an ideal $I$ of $A$ is said to be \emph{$G$-invariant} if 
$\alpha_g(I\cap D_{g^{-1}}) \subseteq I$ 
holds for each $g\in G$. 
If $A$ and $\{0\}$ are the only $G$-invariant ideals of $A$, 
then $A$ is said to be \emph{$G$-simple}.
\end{defn}

\begin{lem}\label{SimpleImpliesGSimple}
If $\alpha$ is a partial action of a group $G$
on a ring $A$ such that for each $g \in G$, the ring
$D_g$ has local units, then $\SkewRing$ is graded simple 
if and only if $A$ is $G$-simple.
\end{lem}

\begin{proof}
We begin by showing the ''only if'' statement. Suppose that $\SkewRing$ is graded simple.
Let $I$ be a non-zero $G$-invariant ideal of $A$.
Define $I \star_\alpha G$ to be the set 
of all finite sums of the form $\sum_{g\in G} a_g \delta_g$, where $a_g\in I \cap D_g$, for $g\in G$.
Note that $I \star_\alpha G$ is a non-zero two-sided graded ideal of $\SkewRing$.
Hence, $I \star_\alpha G = \SkewRing$. In particular, 
$A \delta_e \subseteq I \star_\alpha G$ which shows that $I \subseteq A \subseteq I$.
We conclude that $I=A$. Thus, $A$ is $G$-simple.

Now we show the ''if'' statement. Suppose that $A$ is $G$-simple.
Let $J$ be a non-zero graded ideal of $\SkewRing$.
We claim that $J_e = J \cap A$ is a non-zero $G$-invariant ideal of $A$. 
If we assume that the claim holds, then 
$A = J_e = A \cap J \subseteq J$ from which it follows that $J = \SkewRing$.
Now we show the claim.
First we show that $J_e$ is non-zero.
Since $J$ is non-zero, there is $g \in G$
and a non-zero $a_g \in D_g$ with $a_g \delta_g \in J$.
Let $b_{g^{-1}} \in D_{g^{-1}}$ be a local 
unit for $\alpha_{g^{-1}}(a_g)$.
Then $J \ni a_g \delta_g b_{g^{-1}} \delta_{g^{-1}} =
\alpha_g( \alpha_{g^{-1}}(a_g) b_{g^{-1}} ) \delta_e =
\alpha_g( \alpha_{g^{-1}}(a_g)) \delta_e =
a_g \delta_e$ which is non-zero.
Now we show that $J_e$ is $G$-invariant.
Take $g\in G$ and $a \in J_e \cap D_{g^{-1}}$.  
Let $c_g\in D_g$ be such that $\alpha_{g^{-1}}(c_g)$ is a local unit for $a$.
Then
$\alpha_g(a)u_e=\alpha_g(\alpha_{g^{-1}}(c_g) a )u_e 
= c_g \delta_g a \delta_{g^{-1}} \in J$.
\end{proof}

The following result generalizes the recent result by 
D. Gon\c{c}alves \cite[Theorem 2.5]{Goncalves13} from the 
case when $G$ is abelian to the case when $G$ is a hypercentral group.
Moreover, our result shows that it is enough 
to consider the center of \emph{one} corner.

\begin{thm}\label{Simplicity}
Suppose that $\alpha$ is a partial action of a hypercentral group $G$
on a ring $A$ such that for each $g \in G$, the ring
$D_g$ has local units. Let $E$
denote a set of local units for $A$.
The following three assertions are equivalent:
\begin{enumerate}[{\rm (i)}]
	\item $\SkewRing$ is a simple ring;
	\item $A$ is $G$-simple and the center of the corner subring
	$f\delta_e (\SkewRing) f\delta_e$ is a field, for some non-zero $f\in E$;
	\item $A$ is $G$-simple and the center of the corner subring
	$f\delta_e (\SkewRing) f\delta_e$ is a field, for each non-zero $f\in E$.
\end{enumerate}
\end{thm}

\begin{proof}
This follows immediately from Lemma \ref{SimpleImpliesGSimple} and Theorem \ref{maintheorem}.
\end{proof}

\begin{rem}\label{partialcrossedproducts}
Analogously to the way in which Lemma \ref{SimpleImpliesGSimple} is proven,
one can show, under some mild assumptions, that the crossed product by a 
twisted partial action \cite{DokuExelSimon}, denoted by $A \star_\alpha^w G$, 
is graded simple if and only if $A$ is $G$-simple.
Thereby, 
\cite[Theorem 2.25]{BCS} which was recently observed by A. Baraviera, W. Cortes and M. Soares,
can immediately be retrieved as a corollary to \cite[Theorem 5]{jespers1989}.
In fact, by E. Jespers' result (Theorem \ref{jespers}),
we note that \cite[Theorem 2.25]{BCS} holds for $A \star_\alpha^w G$
even when $G$ is a hypercentral (not necessarily abelian) group.
\end{rem}

\section*{Acknowledgements}
The authors are greatful to an anonymous reviewer for
providing insightful comments and suggestions on previous versions of this manuscript. 
The second author was partially supported by The Swedish Research Council (repatriation grant no. 2012-6113).

\end{document}